\pdfoutput=1
\documentclass[11pt]{article}
\usepackage{fullpage}
\usepackage{amsmath,amssymb,amsthm,mathtools,enumerate}
\usepackage[mathscr]{eucal}
\usepackage{comment}
\usepackage[usenames]{color}
\usepackage{subcaption}
\usepackage[pdftex]{graphicx} 
\usepackage{bbm}

\newcommand{\E}{\mathbb{E}}
\newcommand{\PP}{\mathbb{P}}
\newcommand{\R}{\mathbb{R}}

\newcommand{\uS}{\mathbb{S}}
\newcommand{\One}{\mathbbm{1}}

\newcommand{\sF}{\mathscr{F}}

\newcommand{\inner}[2]{#1 \cdot #2}

\newtheorem{theorem}{Theorem}[section]
\newtheorem{lemma}[theorem]{Lemma}

\title{Convergence of the randomized Kaczmarz method for phase retrieval
}
\author{Halyun Jeong\footnote{email: \tt{halyun@cims.nyu.edu}}\, 
and 
C.~Sinan G\"unt\"urk\footnote{email: \tt{gunturk@cims.nyu.edu}}
 \\ Courant Institute, NYU}
\date{June 30, 2017; revised July 16, 2017}
\begin{document}

\maketitle

\begin{abstract}
The classical Kaczmarz iteration and its randomized variants are popular tools for fast inversion of linear overdetermined systems. This method extends naturally to the setting of the phase retrieval problem via substituting at each iteration the phase of any measurement of the available approximate solution for the unknown phase of the measurement of the true solution. Despite the simplicity of the method, rigorous convergence guarantees that are available for the classical linear setting have not been established so far for the phase retrieval setting. In this short note, we provide a convergence result for the randomized Kaczmarz method for phase retrieval in $\R^d$. We show that with high probability a random measurement system of size $m \asymp d$ will be admissible for this method in the sense that convergence in the mean square sense is guaranteed with any prescribed probability. The convergence is exponential and comparable to the linear setting.
\end{abstract}

\section{Introduction}

The classical Kaczmarz iteration is a popular and convenient method for the recovery of any real or complex $d$-dimensional vector $x$ from a collection of sufficient linear measurements $y_i:=\inner{x}{\phi_i}$, $i=1,\dots,m$, where $\inner{u}{v}$ denotes the Euclidean inner product of $u$ and $v$. Starting with any initial point $x_0$, the algorithm produces a succession of iterates $(x_k)_0^\infty$ defined by 
\begin{equation}\label{Kaczmarz-orig}
x_{k+1}  = x_k + \left( y_{t} - \inner{x_k}{\phi_{t}} \right) \frac{\phi_{t}}{\|\phi_{t} \|^2},
\end{equation}
where $t:=t(k) \in \{ 1,...,m \}$ is the index of the selected vector (and the corresponding measurement) at time $k$.
This equation has a simple interpretation: $x_{k+1}$ is the orthogonal projection of $x_k$ on the solution hyperplane $\{u: \inner{u}{\phi_{t}} = y_{t}\}$. In other words, the update $x_{k+1}-x_k$ is the orthogonal projection of the error $x-x_k$ on the chosen direction $\phi_{t}$. Kaczmarz's original scheme cycles through the indices periodically, but it has been shown that random selection generally yields faster convergence. For this and other results, see \cite{SV, NT, NSW, CP}.

This method can be adapted in a straightforward manner to the phase retrieval problem where we only have access to the intensities $\{|y_i|\}_{i=1}^m$: By simply using the sign (phase) of the approximate measurement $\inner{x_k}{\phi_{t}}$ in place of that of $y_{t}$, we get the {\em phase-adapting Kaczmarz iteration}
\begin{equation}\label{Kaczmarz-phase}
x_{k+1}  = x_k + \left( \sigma(\inner{x_k}{\phi_{t}})|y_{t}| - \inner{x_k}{\phi_{t}} \right) \frac{\phi_{t}}{\|\phi_{t} \|^2},
\end{equation}
where $\sigma(w)$ is the sign (or phase) of the scalar $w$, defined by the relation $w = \sigma(w)|w|$. We will assume the convention that $\sigma(0)=1$.
This method has been proposed by various authors (e.g. \cite{Wei, GL}) and has been observed to perform well in practice. For general theory and some other main approaches to the phase retrieval problem, such as PhaseLift and PhaseCut, see \cite{EM2014,CSV2013,WdM2015}.

Intuitively, this scheme has the biggest chance of success if the iterates can be guaranteed to stay reasonably close to one of the solutions of the phaseless equations so that the approximate signs $\sigma(\inner{x_k}{\phi_{t}})$ have a chance to frequently match (or approximate, in the complex case) the true signs and make progress. Each time there is a phase mismatch, the iterate gets an update in the wrong direction, so it is important that this event does not happen too frequently. Hence, unlike the linear classical Kaczmarz scheme \eqref{Kaczmarz-orig} which is not susceptible to the initial condition, a good initialization is needed for the nonlinear phase-adapting version \eqref{Kaczmarz-phase}. There are now good methods for this, such as the truncated spectral initialization \cite{CC}.

\subsection{Contribution}

This paper will be about the real case, i.e. both $x$ and the $\phi_i$ are in $\R^d$. Without loss of generality, we assume that the $\phi_i$ are of unit norm, since we can always run the iteration \eqref{Kaczmarz-phase} with normalized vectors $\hat \phi_i := \phi_i/\|\phi_i\|$ and intensity measurements $|\hat y_i|:= |y_i|/\|\phi_i\|$. Hence we will work with the iteration

\begin{equation}\label{Kaczmarz-phase-normalized}
x_{k+1}  = x_k + \left( \sigma(\inner{x_k}{\phi_{t}})|\inner{x}{\phi_{t}}| - \inner{x_k}{\phi_{t}} \right) \phi_{t}.
\end{equation}

There will be two sources of randomness in this paper. The first and the primary source of randomness is the following:
Given any measurement system $\Phi:=(\phi_1,\dots,\phi_m)$, we will assume that the indices $t$ are chosen uniformly and independently from $\{1,\dots,m\}$. We will call the resulting method {\em phase-adapting randomized Kaczmarz iteration}, irrespective of how $\Phi$ may have been chosen.  
In Section \ref{sec-admissibility}, we present a certain deterministic condition on $\Phi$ called ``$\delta$-admissibility''(which consists of four individual properties), and show that with a $\delta$-admissible $\Phi$ (and for a sufficiently small $\delta$), if the starting relative error is less than $\delta$, then after one iteration the error shrinks in conditional expectation (with respect to the random choice of $t$). We then carry out a probabilistic analysis of convergence in Section \ref{sec-prob-analysis} via ``drift analysis'' and ``hitting-time'' bounds.

The secondary source of randomness will come into play when we want to show that most measurement systems are $\delta$-admissible in the $m \asymp d$ regime. To achieve this, we will assume that the $\phi_i$ are chosen independently from the uniform distribution on the unit sphere $\uS^{d-1}$ in $\R^d$. The standard Gaussian distribution on $\R^d$ can also be used.

For convenience, we state here a summarized theorem combining these two types of randomness. Individual (and stronger) results are stated separately in Sections \ref{sec-admissibility} and \ref{sec-prob-analysis}. We use the notation
$$\mathrm{dist}(u,v):= \min(\|u-v\|, \|u+v\|)$$
to denote the distance between $u$ and $v$ up to a global phase.

\begin{theorem}\label{main-theorem}
Let $\phi_1,\dots,\phi_m$ be chosen independently and uniformly on $\uS^{d-1}$.
There exist absolute positive constants $C$, $c$, and $\delta_0$ such that if $m \geq C d$, then 
with probability $1-\exp(-cm)$ the system $\Phi:=(\phi_1,\dots,\phi_m)$ satisfies the following property: 

For any $0<\varepsilon <1$, if the phase-adapting randomized Kaczmarz method with respect to $\Phi$ is applied to any initial point $x_0$ satisfying the relative error bound
$$\frac{\mathrm{dist}(x,x_0)}{\|x\|} \leq \delta_0 \varepsilon,$$
then the stability event
$$\Sigma:=\left\{\frac{\mathrm{dist}(x,x_k)}{\|x\|} \leq \delta_0~~ \text{for all}~~ k\geq 1\right \}$$
holds with probability at least $1-\varepsilon^2$, and conditioned on this event the expected squared error decays exponentially. More precisely, we have
$$\E \left[\mathrm{dist}^2(x,x_k) \One_\Sigma \right] \leq e^{-k/4d} \mathrm{dist}^2(x,x_0)$$
for all $k\geq 1$.
\end{theorem}

We prove this theorem at the end of Section \ref{sec-prob-analysis}. Some remarks are in order:

\begin{itemize}
\item As is the case for the randomized Kaczmarz method for linear inverse problems, the exponential convergence of $x_k$ to $x$ is achieved in the mean-squared sense. However, an important distinction is that this is conditional on a stability event. (In the linear case, this event is automatic due to the fact that error decreases deterministically.) We handle this problem using methods that are known as ``drift analysis'' (see \cite{HB}).

\item The above stated probability lower bound for the stability event is not tight. Furthermore, our preliminary calculations suggest that the methods of this paper can be extended to achieve an improved probabilistic guarantee of the form $1-O(\varepsilon^{2p})$ for any {\em fixed} $p \geq 1$. For the sake of exposition we do not pursue this extension in this manuscript.

\item We have left out performance guarantees regarding the initialization procedure from the above theorem because we have no new results to offer here. One may simply use the truncated spectral method \cite{CC} which is capable of providing the kind of guarantee that is compatible with the above theorem in that for any accuracy guarantee it can operate in the regime $m \asymp d$ and succeed with probability $1-\exp(-\Omega(m))$. 

\end{itemize}

{\bf Note for the revision:} We would like to note here that simultaneously with the initial posting of this paper, Y.~Shuo Tan and R.~Vershynin posted a manuscript (see \cite{Shuo_etal}) on the randomized Kaczmarz method for phase retrieval, with results that are somewhat similar to ours, but established using different methods. Subsequently, we were also informed that Zhang et al.~\cite{Zhang_etal} had previously established a conditional error contractivity result for the Gaussian measurement model and using the so-called ``reshaped Wirtinger flow'' method.  

\section{Basic relations}

Let $z_k:=x-x_k$. Then \eqref{Kaczmarz-phase-normalized} can be rewritten as 
\begin{eqnarray}
z_{k+1} &=& z_k - \left( \sigma(\inner{x_k}{\phi_{t}})|\inner{x}{\phi_{t}}| - \inner{x_k}{\phi_{t}} \right) \phi_{t} \nonumber \\
& = & 
\label{error-iteration}
z_k - (\inner{z_k}{\phi_{t}})\phi_{t} + \left[\sigma(\inner{x}{\phi_t})- \sigma(\inner{x_k}{\phi_{t}}) \right]|\inner{x}{\phi_{t}}|\phi_{t}. 
\end{eqnarray}
Since $z_k - (\inner{z_k}{\phi_{t}})\phi_{t}$ and $\phi_t$ are orthogonal, we obtain
\begin{eqnarray}
 \|z_{k+1}\|^2 & = & \|z_k - (\inner{z_k}{\phi_{t}})\phi_{t}\|^2 + \left|\sigma(\inner{x}{\phi_t})- \sigma(\inner{x_k}{\phi_{t}}) \right|^2|\inner{x}{\phi_{t}}|^2 \nonumber \\
 & = & 
 \label{error-iteration-norm}
 \|z_k\|^2 - |\inner{z_k}{\phi_{t}}|^2 + \left|\sigma(\inner{x}{\phi_t})- \sigma(\inner{x_k}{\phi_{t}}) \right|^2|\inner{x}{\phi_{t}}|^2. 
\end{eqnarray}
When $\inner{x}{\phi_t}$ and $\inner{x_k}{\phi_{t}}$ have opposite signs we have
$|\inner{x}{\phi_t}| \leq |\inner{(x-x_k)}{\phi_t}|$ so that 
$$ \left|\sigma(\inner{x}{\phi_t})- \sigma(\inner{x_k}{\phi_{t}}) \right|
|\inner{x}{\phi_t}| \leq 
\left|\sigma(\inner{x}{\phi_t})- \sigma(\inner{x_k}{\phi_{t}}) \right|
|\inner{z_k}{\phi_t}|$$
is always valid. Hence \eqref{error-iteration-norm} implies
\begin{equation} \label{error-iteration-norm-bound}
 \|z_{k+1}\|^2 \leq \|z_k\|^2 + \left[\left|\sigma(\inner{x}{\phi_t})- \sigma(\inner{x_k}{\phi_{t}}) \right|^2-1\right]|\inner{z_k}{\phi_{t}}|^2. 
\end{equation}

Note that \eqref{Kaczmarz-phase-normalized} is invariant under the transformation $x \mapsto -x$. Hence we actually have
\begin{equation}
 \|\pm x - x_{k+1}\|^2 \leq \|\pm x-x_k\|^2 + \left[\left|\sigma(\inner{\pm x}{\phi_t})- \sigma(\inner{x_k}{\phi_{t}}) \right|^2-1\right]|\inner{(\pm x-x_k)}{\phi_{t}}|^2,
\end{equation}
i.e. the analysis is identical for $x$ and $-x$.
For convenience of notation and without loss of generality we will work to analyze $\|x-x_k\|$ and make our initial condition assumption on $\|x-x_0\|$. 

\subsection{Heuristic for convergence}\label{heuristic}
Let $\phi$ be uniformly distributed on $\uS^{d-1}$.
It is a standard fact that 
$$ \E ~|\inner{z}{\phi}|^2 = \frac{1}{d}\|z\|^2,$$
and an easy calculation (see Appendix) yields
$$ \E ~|\inner{z}{\phi}|^4 = \frac{3}{d(d+2)}\|z\|^4.$$
It can also be checked easily that for any two nonzero $x$ and $y$ we have
$$ \PP\{\sigma(\inner{x}{\phi}) \not= \sigma(\inner{y}{\phi})\} =
\frac{1}{\pi}\theta_{x,y} =: d(\hat x,\hat y),$$
where $\theta_{x,y}\in[0,\pi]$ is the angle between $x$ and $y$, and therefore
$d(\hat x,\hat y)$ is the normalized geodesic distance on $\uS^{d-1}$ between $\hat x$ and $\hat y$. Hence, by Cauchy-Schwarz inequality, we obtain 
\begin{eqnarray}
\E \left|\sigma(\inner{x}{\phi})- \sigma(\inner{y}{\phi}) \right|^2|\inner{(x-y)}{\phi}|^2
& \leq & 4 \left (\PP\{\sigma(\inner{x}{\phi}) \not= \sigma(\inner{y}{\phi})\}\right)^{1/2} \left(\E |\inner{(x-y)}{\phi}|^4\right)^{1/2} \nonumber \\
& \leq & \frac{4}{d} \left(\frac{3\theta_{x,y}}{\pi}\right)^{1/2} \|x - y\|^2.
\end{eqnarray}
Hence, if $\theta_{x,y}$ is sufficiently small (e.g., less than $1/64$), then 
$$ \E \left(\left|\sigma(\inner{x}{\phi})- \sigma(\inner{y}{\phi}) \right|^2-1\right)|\inner{(x-y)}{\phi}|^2
\leq -\frac{1}{2d} \|x - y\|^2.$$

Guided by these calculations, we turn to the error bound \eqref{error-iteration-norm-bound}. We see that if $\theta_{x,x_k}$ is sufficiently small (which, for a fixed $x$, would be guaranteed by a sufficiently small $z_k$) and if we were to choose each $\phi_t$ uniformly and independently on the unit sphere, then we would have 
$$ \E \left [\|z_{k+1}\|^2 \Big | \sF_k,\theta_{x,x_k} < \frac{1}{64}\right] \leq \left(1-\frac{1}{2d}\right) \|z_k\|^2,$$
where $\sF_k$ is the sigma-algebra generated by $\phi_{t(0)},\dots,\phi_{t(k{-}1)}$, and for any event $E$, $\{\sF_k,E\}$ is the sigma-algebra in $E$ formed by intersecting elements of $\sF_k$ with $E$. 

Hence the stochastic process $(\|z_k\|^2)_{k=0}^\infty$ is contractive in conditional expectation which is also conditional on the size of $\|z_k\|$. Without the size condition on $\|z_k\|$, the analysis would have been fairly straightforward, similar to the situation of the randomized Kaczmarz iteration for linear systems. As we will see, this condition makes the task non-trivial.

However, we must also establish a similar contractivity result (conditional and in expectation) for the actual random model used in this paper, i.e., when $\phi_t$ is chosen uniformly from a fixed collection $\Phi:=(\phi_1,\dots,\phi_m)$. This collection itself may also have been chosen randomly, though with the above observation we can now define certain deterministic properties of $\Phi$ that are needed for the algorithm to work.

\section{Admissible measurement systems} \label{sec-admissibility}

Let $\delta \in (0,1)$ and $\Phi:=(\phi_1,\dots,\phi_m)$ be a given collection of nonzero vectors in $\R^d$.
Following \cite{PV}, we say that $\Phi$, or more appropriately, the linear hyperspaces $(\phi_1^\perp,\dots,\phi_m^\perp)$ produce a $\delta$-uniform tessellation of $\uS^{d-1}$ if for all $x$ and $y$ in $\uS^{d-1}$, we have
\begin{equation}\label{delta-uniform}
\left| \frac{1}{m}\text{card}\big \{1\leq i \leq m : \sigma(\inner{x}{\phi_i}) \not= \sigma(\inner{y}{\phi_i}) \big \}
 - d(x,y) \right | < \delta.
\end{equation}
Then by Theorem 1.2 of \cite{PV}, there exists two positive absolute constants $C$ and $c$ such that if $m \geq C \delta^{-6}d$ and the $\phi_i$ are chosen independently from the uniform distribution on $\uS^{d-1}$, then with probability at least $1-2\exp(-c\delta^2 m)$, we get a $\delta$-uniform tessellation of $\uS^{d-1}$.

If $\phi$ is chosen from the collection $\Phi:=(\phi_1,\dots,\phi_m)$ uniformly at random and $f$ is any function on $\Phi$, then we define the empirical mean
$$ \E_{\phi\sim \Phi} ~ f(\phi) := \frac{1}{m} \sum_{i=1}^m f(\phi_i). $$
With a $\Phi$ that yields a $\delta$-uniform tessellation, we have that the empirical mean $\E_{\phi\sim \Phi} {\bf 1}_{\{\sigma(\inner{x}{\phi}) \not= \sigma(\inner{y}{\phi})\}}$ is within $\delta$ of the ensemble mean $d(\hat x,\hat y)$. The upper part of this bound obviously yields
\begin{equation}\label{delta-uniform-implies}
 \E_{\phi\sim\Phi} ~ \One_{\left \{\sigma(\inner{x}{\phi})\not= \sigma(\inner{y}{\phi}) \right \}} \leq \delta + d(\hat x,\hat y)~~~\mbox{ for all } x,y.  
\end{equation}

The above result provides a pathway for mimicking the argument in Section \ref{heuristic} with $\E$ replaced by $\E_{\phi\sim \Phi}$.
Under the same random model for $\Phi$, a useful concentration result (i.e. for the regime $m \asymp d$) holds for the empirical mean $\E_{\phi\sim \Phi} |\inner{z}{\phi}|^2$. Indeed as it follows from \cite[Theorem 5.39]{Vershynin}, there exist absolute positive constants $C$ and $c$ such that for $m \geq Cd$ and with probability $1-\exp(-cm)$ we have
\begin{equation}\label{2-lower-upper-bound}
\frac{\|z\|^2}{2d} \leq \E_{\phi\sim \Phi} ~|\inner{z}{\phi}|^2 \leq \frac{3\|z\|^2}{2d} ~~~\mbox{ for all } z.
\end{equation}
(If desired, the constants $1/2$ and $3/2$ can be chosen closer to $1$ without changing the form of this statement.)

In order to continue on the same path, one would wish to have $\E_{\phi\sim \Phi} ~|\inner{z}{\phi}|^4 \lesssim \|z\|^4/d^2$ with high probability. As it turns out,\footnote{We thank Y. Shuo Tan and R. Vershynin for bringing this fact to our attention.} this is impossible in the regime $m \asymp d$.
We will circumvent this obstacle by tightening the Cauchy-Schwarz argument of Section \ref{heuristic}: In order to do this, we will invoke \eqref{delta-uniform-implies} coupled with the Cauchy-Schwarz inequality only in the event $|\inner{z}{\phi}|^2$ does not exceed a fixed multiple of its mean value $\|z\|^2/d$, and show that the above desirable upper bound is then achievable with high probability. At the same time, we will show that the second moment contribution from the large values is in fact small, so in this event we will only invoke the trivial bound on $\left|\sigma(\inner{x}{\phi})- \sigma(\inner{y}{\phi}) \right|^2$. 

To this end, given $\delta \in (0,1)$, consider the alternative weaker conditions
\begin{equation}\label{4-truncate-upper-bound}
\E_{\phi\sim \Phi} ~|\inner{z}{\phi}|^4\, \One_{\left \{|\inner{z}{\phi}|^2\leq \|z\|^2/\delta d\right \}} \leq \frac{4\|z\|^4}{d^2} ~~~\mbox{ for all } z,
\end{equation}
and  
\begin{equation}\label{2-truncate-upper-bound}
\E_{\phi\sim \Phi} ~|\inner{z}{\phi}|^2 \,\One_{\left \{|\inner{z}{\phi}|^2> \|z\|^2/\delta d\right \}} \leq \frac{4\delta \|z\|^2}{d} ~~~\mbox{ for all } z.
\end{equation}
We will say that $\Phi$ is {\em $\delta$-admissible} if all of the four conditions \eqref{delta-uniform-implies}, \eqref{2-lower-upper-bound}, \eqref{4-truncate-upper-bound}, and \eqref{2-truncate-upper-bound} hold. Note that all of these are deterministic conditions on $\Phi$. We will show in Lemma \ref{admissible-whp} that a random measurement system $\Phi$ is $\delta$-admissible with high probability when $m \geq C d$, but first let us show how these two alternative conditions are used instead of a bound on $\E_{\phi\sim \Phi} ~|\inner{z}{\phi}|^4$. Suppose $\Phi$ is $\delta$-admissible. 
Noting that $\left|\sigma(\inner{x}{\phi})- \sigma(\inner{y}{\phi}) \right|^2 /4 = \One_{\left \{\sigma(\inner{x}{\phi})\not= \sigma(\inner{y}{\phi}) \right \}}$, we have
\begin{align}
& \E_{\phi\sim \Phi}  ~\left|\sigma(\inner{x}{\phi})- \sigma(\inner{y}{\phi}) \right|^2|\inner{(x-y)}{\phi}|^2
\nonumber \\
& \qquad \qquad \qquad \qquad = ~ 4 ~ \E_{\phi\sim\Phi} ~ \One_{\left \{\sigma(\inner{x}{\phi})\not= \sigma(\inner{y}{\phi}) \right \}} 
~|\inner{(x-y)}{\phi}|^2\,\One_{\left \{|\inner{(x-y)}{\phi}|^2\leq \|x-y\|^2/\delta d\right \}}  \nonumber \\
& \qquad \qquad \qquad \qquad \qquad \qquad + 4 ~ \E_{\phi\sim\Phi} ~ \One_{\left \{\sigma(\inner{x}{\phi})\not= \sigma(\inner{y}{\phi}) \right \}} 
~|\inner{(x-y)}{\phi}|^2\,\One_{\left \{|\inner{(x-y)}{\phi}|^2> \|x-y\|^2/\delta d\right \}} 
\nonumber \\
& \qquad \qquad \qquad \qquad \leq ~ 4 \left(\delta + d(\hat x,\hat y)\right)^{1/2}
\Big(\E_{\phi\sim\Phi} ~|\inner{(x-y)}{\phi}|^4\,\One_{\left \{|\inner{(x-y)}{\phi}|^2\leq \|x-y\|^2/\delta d\right \}}\Big)^{1/2}  \nonumber \\
& \qquad \qquad \qquad \qquad \qquad \qquad + 4 ~ \E_{\phi\sim\Phi} 
~|\inner{(x-y)}{\phi}|^2\,\One_{\left \{|\inner{(x-y)}{\phi}|^2> \|x-y\|^2/\delta d\right \}}. 
\end{align}
Invoking \eqref{4-truncate-upper-bound} and \eqref{2-truncate-upper-bound}, we get that
\begin{eqnarray}
 \E_{\phi\sim \Phi}  ~\left|\sigma(\inner{x}{\phi})- \sigma(\inner{y}{\phi}) \right|^2|\inner{(x-y)}{\phi}|^2
 & \leq & 
 \left(8 (\delta + d(\hat x,\hat y))^{1/2} + 16 \delta \right) \frac{\|x-y\|^2}{d} \nonumber \\
 & \leq & \frac{1}{4d} \|x - y\|^2,
\end{eqnarray}
provided $d(\hat x, \hat y)\leq \delta$ and $\delta$ is sufficiently small (e.g. $8\sqrt{2\delta} + 16\delta \leq 1/4$). Hence, together with the lower bound of \eqref{2-lower-upper-bound}, we have 
$$ \E_{\phi\sim \Phi} ~\left(\left|\sigma(\inner{x}{\phi})- \sigma(\inner{y}{\phi}) \right|^2-1\right)|\inner{(x-y)}{\phi}|^2
\leq -\frac{1}{4d} \|x - y\|^2$$
and therefore
$$ \E \left [\|z_{k+1}\|^2 \Big | \sF_k,d(\hat x,\hat x_k) \leq \delta\right] \leq \Big(1-\frac{1}{4d}\Big) \|z_k\|^2,$$
where again $\sF_k$ is the sigma-algebra generated by $\phi_{t(0)},\dots,\phi_{t(k{-}1)}$.

At this point, it will be helpful to replace the condition $d(\hat x,\hat x_k) \leq \delta$ by a size condition on $z_k$. Note that for any two nonzero vectors $x$ and $y$, we have 
$$
d(\hat x, \hat y) \leq \frac{1}{2} \|\hat x - \hat y\| = 
\frac{1}{2} \left \| \frac{x}{\|x\|} - \frac{y}{\|y\|} \right \|
= \frac{\big \| (x - y)\|y\| + y(\|y\|-\|x\|) \big \|}{2\|x\| \|y\|} \leq \frac{\|x-y\|}{\|x\|}$$
so that the condition $\|z_k\|\leq \delta \|x\|$ implies $d(\hat x, \hat x_k) \leq \delta$. Therefore we have
\begin{equation}
\E \left [\|z_{k+1}\|^2 \Big | \sF_k,\|z_k\| \leq \delta \|x\| \right] \leq  \Big(1-\frac{1}{4d}\Big) \|z_k\|^2.
\end{equation}

With the above discussion we have established the following result:
\begin{lemma}\label{admissible-implies-conditional-contraction}
There exists $\delta_0>0$ such that, if $\delta\leq \delta_0$ and $\Phi$ is $\delta$-admissible, then 
\begin{equation}\label{conditional-contraction}
\E \left [\|z_{k+1}\|^2 \Big | \sF_k,\|z_k\| \leq b \right] \leq  \rho \|z_k\|^2~~~\mbox{ for all } k\geq 0,
\end{equation}
where $b:=\delta \|x\|$ and  $\rho := 1-\frac{1}{4d}$.
\end{lemma}

We now show that a random $\Phi$ is $\delta$-admissible with high probability in the regime $m \asymp d$.

\begin{lemma}\label{admissible-whp}
For every $\delta \in (0,1)$, there exists positive constants $C$ and $c$ depending only on $\delta$ such that 
if $m\geq Cd$, then a random measurement system $\Phi:=(\phi_1,\dots,\phi_m)$ that is chosen independently from the uniform distribution on $\uS^{d-1}$ is $\delta$-admissible with probability at least $1-\exp(-cm)$.
\end{lemma}

\begin{proof}
The property \eqref{delta-uniform-implies} is proven in \cite{PV} and \eqref{2-lower-upper-bound} is covered by 
\cite[Theorem 5.39]{Vershynin}. Hence we only need to establish \eqref{4-truncate-upper-bound} and \eqref{2-truncate-upper-bound}. Note that by homogeneity we may assume $\|z\|=1$. 

We start with \eqref{4-truncate-upper-bound}. As is standard in this type of question, we would like to establish the stated inequality for fixed $z$ first (with high probability) and then use approximation over an $\epsilon$-net of $\uS^{d-1}$ to achieve uniformity over $z$.
However $\One_{\left \{|\inner{z}{\phi}|^2\leq \|z\|^2/\delta d\right \}}$ is a discontinuous function of the random variable $|\inner{z}{\phi}|$, presenting a difficulty for the approximation argument. The solution will follow by incorporating a suitable Lipschitz extension, as also done in \cite{CC}. 

For this purpose, let $\gamma_1:[0,\infty) \to [0,\infty)$ be defined by
\begin{equation}\label{gamma-1}
\gamma_1(s) := \left\{\begin{array}{ll}
s^2,& s \leq \delta^{-1},\\
(2 \delta^{-1} - s)\delta^{-1},& \delta^{-1} < s \leq 2\delta^{-1},\\
0,& 2\delta^{-1} < s.
\end{array}\right.
\end{equation}
Then $\gamma_1$ is a Lipschitz function with Lipschitz constant $2 \delta^{-1}$. Furthermore,
$$
s^2 \chi_{[0,\delta^{-1}]}(s) \leq \gamma_1(s) \leq s^2
$$
so that for any $z$ we have
\begin{eqnarray}\label{gamma_1-initial}
\E_{\phi\sim \Phi} ~|\inner{z}{\phi}|^4\, \One_{\left \{|\inner{z}{\phi}|^2\leq 1/\delta d\right \}} 
& = & \frac{1}{d^2} \,\E_{\phi\sim \Phi} ~|\inner{z}{(\sqrt{d} \phi)}|^4\, \chi_{[0,\delta^{-1}]}(|\inner{z}{(\sqrt{d} \phi)}|^2) \nonumber \\
& \leq & \frac{1}{d^2} \,\E_{\phi\sim \Phi} ~\gamma_1(|\inner{z}{(\sqrt{d} \phi)}|^2).
\end{eqnarray}

Now, let $\phi$ denote the random vector uniformly distributed on $\uS^{d-1}$ so that $\sqrt{d} \phi$ is a {\em spherical} random vector in $\R^d$ (see \cite[Section 5.2.5]{Vershynin}). Let
$\| \cdot \|_{\psi_1}$ stand for the sub-exponential norm and
$\| \cdot \|_{\psi_2}$ the sub-Gaussian norm (see \cite[Section 5.2.3 and 5.2.4]{Vershynin}).
Noting that $\gamma_1(s)\leq \delta^{-1}s$, we have
$$ \left \| \gamma_1(|\inner{z}{(\sqrt{d} \phi)}|^2) \right \|_{\psi_1}
\leq \delta^{-1} \left \| |\inner{z}{(\sqrt{d} \phi)}|^2 \right \|_{\psi_1}
\leq 2 \delta^{-1} \left \| |\inner{z}{(\sqrt{d} \phi)}| \right \|^2_{\psi_2}
\lesssim \delta^{-1}
$$
where in the second step we have used \cite[Lemma 5.14]{Vershynin}) and in the
last step the fact that the sub-Gaussian norm of a spherical random vector is bounded by an absolute constant (see \cite[Section 5.2.5]{Vershynin}; a direct computation is also possible).

Hence, by the Bernstein-type inequality \cite[Proposition 5.16]{Vershynin}, there is an absolute constant $c_1>0$ such that for any $t>0$ we have, with probability at least $1-\exp(-c_1\min(\delta t, \delta^2 t^2)m)$,
\begin{eqnarray}\label{gamma_1-concentration}
\E_{\phi\sim \Phi} ~\gamma_1(|\inner{z}{(\sqrt{d} \phi)}|^2)
& \leq &  \E ~\gamma_1(|\inner{z}{(\sqrt{d} \phi)}|^2) ~+~ t \nonumber \\
& \leq &  \E ~|\inner{z}{(\sqrt{d} \phi)}|^4 ~+~ t \nonumber \\
& \leq & 3 ~+~ t,
\end{eqnarray}
where in the second step we have used $\gamma_1(s) \leq s^2$ instead.

Now pick an $\epsilon$-net $\mathcal{N}$ of the unit sphere $\uS^{d-1}$ of cardinality at most $(3/\epsilon)^d$ where $\epsilon < 1$. For each $z' \in \uS^{d-1}$ and $z \in \mathcal{N}$ such that $\|z'-z\|<\epsilon$, we have
\begin{eqnarray}\label{gamma_1-Lipschitz}
\E_{\phi\sim \Phi} ~\left| \gamma_1(|\inner{z'}{(\sqrt{d} \phi)}|^2) - \gamma_1(|\inner{z}{(\sqrt{d} \phi)}|^2)\right |
& \leq &  2\delta^{-1} \E_{\phi\sim \Phi} ~ \left| |\inner{z'}{(\sqrt{d} \phi)}|^2 - |\inner{z}{(\sqrt{d} \phi)}|^2\right | \nonumber \\
& = & 2\delta^{-1} \E_{\phi\sim \Phi} ~ \left |[\inner{(z'-z)}{(\sqrt{d}\phi)}][\inner{(z'+z)}{(\sqrt{d}\phi)}]\right | \nonumber \\
& \leq & 6 \delta^{-1} \epsilon,
\end{eqnarray}
where in the first step we have utilized the Lipschitz continuity of $\gamma_1$, and in the last step Cauchy-Schwarz inequality coupled with the upper bound of \eqref{2-lower-upper-bound}.
Combining \eqref{gamma_1-initial}, \eqref{gamma_1-concentration}, and \eqref{gamma_1-Lipschitz}, we find that with probability at least $1-(3/\epsilon)^d\exp(-c_1 \min(\delta t, \delta^2 t^2)m)$, we have
$$
\E_{\phi\sim \Phi} ~|\inner{z'}{\phi}|^4\, \One_{\left \{|\inner{z'}{\phi}|^2\leq 1/\delta d\right \}} 
\leq \frac{1}{d^2} (3 + t + 6\delta^{-1}\epsilon)
$$
for every $z' \in \uS^{d-1}$. We may choose $t=1/2$ and $\epsilon = \delta/12$ so that 
$3 + t + 6\delta^{-1}\epsilon = 4$ and therefore \eqref{4-truncate-upper-bound} holds with probability at least $1-\exp(-c_1 \delta^2 m/8)$ provided $c_1 \delta^2 m/8 \geq d \log (36/\delta)$.

We continue with \eqref{2-truncate-upper-bound}. We will use the same method, but with a different Lipschitz function. Let $\gamma_2:[0,\infty) \to [0,\infty)$ be defined by
\begin{equation}\label{gamma-2}
\gamma_2(s) := \left\{\begin{array}{ll}
\delta s^2,& s \leq \delta^{-1}\\
s, & s > \delta^{-1}.\\
\end{array}\right.
\end{equation}
Then $\gamma_2$ is a Lipschitz function that fixes $0$ with Lipschitz constant $2$. We have
$$
s \chi_{(\delta^{-1},\infty)}(s) \leq \gamma_2(s) = \min(\delta s^2, s)
$$
so that for any fixed $z$ we have
\begin{eqnarray}
\E_{\phi\sim \Phi} ~|\inner{z}{\phi}|^2\, \One_{\left \{|\inner{z}{\phi}|^2> 1/\delta d\right \}} 
& = & \frac{1}{d} \,\E_{\phi\sim \Phi} ~|\inner{z}{(\sqrt{d} \phi)}|^2\, \chi_{(\delta^{-1},\infty)}(|\inner{z}{(\sqrt{d} \phi)}|^2) \nonumber \\
& \leq & \frac{1}{d} \,\E_{\phi\sim \Phi} ~\gamma_2(|\inner{z}{(\sqrt{d} \phi)}|^2).
\end{eqnarray}
Noting that $\gamma_2(s) \leq s$, we now have 
$$ \left \| \gamma_2(|\inner{z}{(\sqrt{d} \phi)}|^2) \right \|_{\psi_1} \lesssim 1$$
so that by the Bernstein-type inequality (and reducing the value of $c_1$ if necessary), for any $t\in (0,1)$ we have, with probability at least $1-\exp(-c_1 t^2 m)$,
\begin{eqnarray}
\E_{\phi\sim \Phi} ~\gamma_2(|\inner{z}{(\sqrt{d} \phi)}|^2)
& \leq &  \E ~\gamma_2(|\inner{z}{(\sqrt{d} \phi)}|^2) ~+~ t \nonumber \\
& \leq &  \delta\, \E ~|\inner{z}{(\sqrt{d} \phi)}|^4 ~+~ t \nonumber \\
& \leq & 3\delta ~+~ t,
\end{eqnarray}
where in the second step we have used $\gamma_2(s) \leq \delta s^2$ instead.
We again pick an $\epsilon$-net $\mathcal{N}$ of the unit sphere $\uS^{d-1}$ of cardinality at most $(3/\epsilon)^d$. For each $z' \in \uS^{d-1}$ and $z \in \mathcal{N}$ such that $\|z'-z\|<\epsilon$, this time we have
\begin{eqnarray}
\E_{\phi\sim \Phi} ~\left| \gamma_2(|\inner{z'}{(\sqrt{d} \phi)}|^2) - \gamma_2(|\inner{z}{(\sqrt{d} \phi)}|^2)\right |
& \leq &  2\,\E_{\phi\sim \Phi} ~ \left| |\inner{z'}{(\sqrt{d} \phi)}|^2 - |\inner{z}{(\sqrt{d} \phi)}|^2\right | \nonumber \\
& \leq & 6 \epsilon.
\end{eqnarray}
Hence by the union bound, with probability at least $1-(3/\epsilon)^d\exp(-c_1 t^2 m)$ we have 
$$
\E_{\phi\sim \Phi} ~|\inner{z'}{\phi}|^2\, \One_{\left \{|\inner{z'}{\phi}|^2> 1/\delta d\right \}} 
\leq \frac{1}{d} (3\delta + t + 6 \epsilon)
$$
for every $z' \in \uS^{d-1}$. We may choose $t=\delta/2$ and $\epsilon = \delta/12$ so that 
$3\delta + t + 6\epsilon = 4\delta$ and therefore \eqref{2-truncate-upper-bound} holds with probability at least  $1-\exp(-c_1 \delta^2 m/8)$ provided $c_1 \delta^2 m/8 \geq d \log (36/\delta)$.
\end{proof}

\section{Probabilistic analysis of the error sequence} \label{sec-prob-analysis}

Our goal in this section will be to bound the probability that $\|z_k\|$ exceeds $b$ at some point and to obtain probabilistic guarantees on the exponential decay of $\|z_k\|$. Lemma \ref{admissible-implies-conditional-contraction} uses the randomness present in the selection of $\phi_{t(k)}$ only. To be able to iterate this result recursively we need to condition on the event $\Omega_k:=\{\|z_0\|\leq b, \|z_1\|\leq b,\dots,\|z_k\|\leq b\}$.
We define the ``hitting time''
$$ \tau_b := \min\{ j \geq 0: \|z_j\| > b\}.$$
Hence $\Omega_k$ is the same as $\{\tau_b > k\}$ and the event $\{\tau_b = \infty\}$ means $\|z_k\|\leq b$ for all $k$.
\begin{lemma}\label{full-contraction}
Suppose \eqref{conditional-contraction} holds. Then 
\begin{equation}\label{recursive-contraction}
\E \left [\|z_{k+1}\|^2 \Big | \tau_b > k \right] \leq  \rho~ \E \left [\|z_{k}\|^2 \Big | \tau_b > k{-}1 \right]
\end{equation}
and therefore
\begin{equation}\label{recursive-contraction-outcome}
\E \left [\|z_{k}\|^2 \Big | \tau_b > k{-}1 \right] \leq  \rho^k \|z_0\|^2
\end{equation}
for all $k\geq 0$.
\end{lemma}
\begin{proof}
Note that $\sF_0$ is the trivial sigma-algebra. We have
\begin{eqnarray}
 \E \left [\|z_{k+1}\|^2 \Big | \tau_b > k \right] 
 & = & \E \left [\|z_{k+1}\|^2 \Big | \tau_b > k, \sF_0 \right] \nonumber \\
 & = & \E \left [\E \left [\|z_{k+1}\|^2 \Big | \tau_b > k, \sF_k \right]\Big | \tau_b > k, \sF_0 \right] \nonumber \\
 & \leq & \rho~ \E \left [\E \left [\|z_{k}\|^2 \Big | \tau_b > k, \sF_k \right]\Big | \tau_b > k, \sF_0 \right] \nonumber \\
 & = & \rho~ \E \left [\|z_{k}\|^2 \Big | \tau_b > k \right]. \label{step-1}
\end{eqnarray}
Note also that 
\begin{equation}\label{sandwich}
\E \left [\|z_{k}\|^2 \Big | \tau_b  > k \right] \leq b^2 \leq \E \left [\|z_{k}\|^2 \Big | \tau_b = k \right].
\end{equation}
Hence we have
\begin{eqnarray}
 \E \left [\|z_{k}\|^2 \Big | \tau_b  > k \right] 
 & = & \PP\left(\|z_k\| \leq b \Big| \tau_b > k{-}1\right ) \E \left [\|z_{k}\|^2 \Big | \tau_b  > k \right] \nonumber \\
 & & \quad \quad + ~ \PP\left(\|z_k\| > b \Big| \tau_b > k{-}1\right ) \E \left [\|z_{k}\|^2 \Big | \tau_b  > k \right]  \nonumber \\
 & \leq & \PP\left(\|z_k\| \leq b \Big| \tau_b > k{-}1\right ) \E \left [\|z_{k}\|^2 \Big | \tau_b  > k \right] \nonumber \\
 & & \quad \quad + ~ \PP\left(\|z_k\| > b \Big| \tau_b > k{-}1\right ) \E \left [\|z_{k}\|^2 \Big | \tau_b  = k \right]  \nonumber \\
& = & \PP\left(\|z_k\| \leq b \Big| \tau_b > k{-}1\right ) \E \left [\|z_{k}\|^2 \Big | \|z_k\|\leq b,\tau_b  > k{-}1 \right] \nonumber \\
 & & \quad \quad + ~ \PP\left(\|z_k\| > b \Big| \tau_b > k{-}1\right ) \E \left [\|z_{k}\|^2 \Big | \|z_k\|>b, \tau_b > k{-}1 \right]  \nonumber \\
& = & \E \left [\|z_{k}\|^2 \One_{\{\|z_k\|\leq b\}} \Big | \tau_b  > k{-}1 \right]
+ \E \left [\|z_{k}\|^2 \One_{\{\|z_k\|> b\}}\Big | \tau_b  > k{-}1 \right] \nonumber \\
& = & \E \left [\|z_{k}\|^2 \Big | \tau_b  > k{-}1 \right]. \label{step-2}
\end{eqnarray}
The result now follows by combining \eqref{step-1} and \eqref{step-2}.
\end{proof}

We can now use Lemma \ref{full-contraction} to control (i) the probability of the event that the error $\|z_k\|$ exceeds $b$ at some point (i.e. $\{\tau_b < \infty\}$), and (ii) the expected decay of squared error $\|z_k\|^2$ conditional on the event that the error remains bounded by $b$ (i.e. $\{\tau_b = \infty\}$).

\begin{lemma}\label{bound-error-decay}
Suppose \eqref{conditional-contraction} holds. Then for any $k\geq 1$ we have
\begin{equation}
 \E \left [\|z_{k}\|^2 \Big | \tau_b  = \infty \right] \leq \frac{1}{1-\PP \left (\tau_b < \infty \right)}
 \rho^k \|z_0\|^2. 
\end{equation}
and for any $a>0$
\begin{equation}
 \PP \left ( \|z_k\| \geq \rho^{k/2} a \right) \leq \left(\frac{\|z_0\|}{a}\right)^2 +  \PP \left (\tau_b < \infty \right ).
\end{equation} 
\end{lemma}
\begin{proof}
For the first claim it suffices to observe that $\{\tau_b = \infty\} \subset \{\tau_b > k{-}1\}$ so that
$$
\PP \left (\tau_b = \infty \right) \E \left [\|z_{k}\|^2 \Big | \tau_b  = \infty \right] \leq
\PP \left (\tau_b > k{-}1 \right) \E \left [\|z_{k}\|^2 \Big | \tau_b  > k{-}1 \right].
$$
The result follows by bounding the right hand side of this inequality using \eqref{recursive-contraction-outcome}.

For the second claim, note that 
\begin{eqnarray}
\PP \left ( \|z_k\| \geq \epsilon \right) 
& = & \PP \left ( \|z_k\| \geq \epsilon \Big | \tau_b > k{-}1 \right) \PP \left (\tau_b > k{-}1 \right )
+ \PP \left ( \|z_k\| \geq \epsilon \Big | \tau_b \leq k{-}1 \right) \PP \left (\tau_b \leq k{-}1 \right )
\nonumber \\
& \leq & \PP \left ( \|z_k\| \geq \epsilon \Big | \tau_b > k{-}1 \right) + \PP \left (\tau_b < \infty \right )
\nonumber \\
& \leq & \epsilon^{-2} \E \left [\|z_{k}\|^2 \Big | \tau_b > k{-}1 \right]  + \PP \left (\tau_b < \infty \right ).
\end{eqnarray}
The result follows by setting $\epsilon = \rho^{k/2}a$ and using \eqref{recursive-contraction-outcome} again.
\end{proof}

Next we give a bound on $\PP \left (\tau_b < \infty \right)$.

\begin{lemma}\label{bound-escape-prob}
 Suppose \eqref{conditional-contraction} holds.
 Then 
 \begin{equation}
  b^2\, \PP  \left (\tau_b < \infty \right) + (\rho^{-1}-1)
\sum_{k=1}^\infty \PP\left(\tau_b > k\right )  \E \left [\|z_{k+1}\|^2 \Big | \tau_b  > k \right] \leq \rho \|z_0\|^2, 
 \end{equation}
 in particular we have
 \begin{equation}
  \PP  \left (\tau_b < \infty \right) \leq \rho \left( \frac{\|z_0\|}{b}\right)^2.
 \end{equation}
\end{lemma}

\begin{proof}
We start by noting that 
\begin{eqnarray}
\PP \left( \tau_b > k{-}1 \right) \E \left [\|z_{k}\|^2 \Big | \tau_b  > k{-}1 \right] 
&  = &\quad \PP\left(\|z_k\| > b , \tau_b > k{-}1\right ) \E \left [\|z_{k}\|^2 \Big | \|z_k\| > b, \tau_b  > k{-}1 \right] \nonumber \\
& \quad &+ ~ 
 \PP\left(\|z_k\| \leq b , \tau_b > k{-}1\right ) \E \left [\|z_{k}\|^2 \Big | \|z_k\| \leq b, \tau_b  > k{-}1 \right]
  \nonumber \\
&  \geq &\quad \PP\left(\tau_b = k\right ) b^2 + ~ 
 \PP\left(\tau_b > k\right ) \E \left [\|z_{k}\|^2 \Big | \tau_b  > k \right]
  \nonumber \\
 &  \geq &\quad \PP\left(\tau_b = k\right ) b^2 + ~ 
 \PP\left(\tau_b > k\right ) \rho^{-1} \E \left [\|z_{k+1}\|^2 \Big | \tau_b  > k \right] \nonumber \\
 &  \geq &\quad \PP\left(\tau_b = k\right ) b^2 + ~ 
 \PP\left(\tau_b > k\right )  \E \left [\|z_{k+1}\|^2 \Big | \tau_b  > k \right] \nonumber \\
 & \quad &+ ~ (\rho^{-1}-1) \PP\left(\tau_b > k\right )  \E \left [\|z_{k+1}\|^2 \Big | \tau_b  > k \right] 
\end{eqnarray}
where in the first inequality we have used (the second inequality of) \eqref{sandwich} and in the second equality \eqref{step-1}.
Since $\E \left [\|z_{k}\|^2 \Big | \tau_b  > k{-}1 \right]$ has exponential decay we can sum both sides from $k=1$ to $\infty$. After cancelling the common term (the sum from $k=2$ to $\infty$) we obtain
$$
\PP \left( \tau_b > 0 \right) \E \left [\|z_1\|^2 \Big | \tau_b  > 0 \right] \geq 
b^2 \sum_{k=1}^\infty \PP\left(\tau_b = k\right ) +  (\rho^{-1}-1)
\sum_{k=1}^\infty \PP\left(\tau_b > k\right )  \E \left [\|z_{k+1}\|^2 \Big | \tau_b  > k \right] .
$$
The result follows by noting that the left hand side is bounded above by $\rho \|z_0\|^2$.
\end{proof}

We are now ready to prove Theorem \ref{main-theorem}. 
\paragraph{\it Proof of Theorem \ref{main-theorem}.}
We choose $\delta_0$ as implied by Lemma \ref{admissible-implies-conditional-contraction} and apply  Lemma \ref{admissible-implies-conditional-contraction} and 
\ref{admissible-whp} together where we set the values of $C$ and $c$ according to the choice $\delta = \delta_0$. Hence \eqref{conditional-contraction} is valid with $b = \delta_0 \|x\|$.

Let $\varepsilon \in (0,1)$ and $x_0$ such that $\mathrm{dist}(x,x_0) \leq \delta_0 \varepsilon$ be be given. Let us assume $\|x - x_0 \| \leq \|-x - x_0\|$ so that 
$\mathrm{dist}(x,x_0) = \|x - x_0\|$. Otherwise we replace $x$ by $-x$ below. With \eqref{conditional-contraction}, we have the conclusions of Lemmas \ref{bound-error-decay} and \ref{bound-escape-prob} at our disposal. The stability event is simply equal to $\{ \tau_b = \infty\}$.
We bound $\PP ( \tau_b < \infty )$ by $\rho (\delta_0 \varepsilon \|x\|/\delta_0 \|x\|)^2 \leq \varepsilon^2$. Lemma \ref{bound-error-decay} then yields the decay bound
$$
\E ~\left [\|z_k\|^2 \One_{\{\tau_b = \infty\}} \right] \leq e^{-k/4d} \|z_0\|^2,
$$
where we have used $\rho = 1-\frac{1}{4d} \leq e^{-1/4d}$.
\qed

\section*{Appendix: A simple moment calculation}

Let the random variable $\phi$ be uniformly distributed on $\uS^{d-1}$ and $e_1,\dots,e_d$ be the standard basis of $\R^d$. By the unitary invariance of the uniform distribution on $\uS^{d-1}$, $\E\, |\inner{z}{\phi}|^2$ independent of $z \in \uS^{d-1}$. Hence we have
$$ 1 =  \E\,\|\phi\|^2 = \E\,\sum_{j=1}^d|\inner{e_j}{\phi}|^2 = d \,
 \E\,|\inner{e_1}{\phi}|^2 $$
so that $\E\, |\inner{e_1}{\phi}|^2 = \frac{1}{d}$, and therefore $\E\, |\inner{z}{\phi}|^2 = \frac{\|z\|^2}{d}$ for general $z \in \R^d$.

Let us now consider $\E\, |\inner{z}{\phi}|^4$ which is also independent of $z \in \uS^{d-1}$. Call its common value $\alpha$. It is also clear by symmetry that for $i\not= j$, the quantity
$\E\, |\inner{e_i}{\phi}|^2|\inner{e_j}{\phi}|^2$ is independent of the pair $(i,j)$. Call its common value $\beta$.

We have 
$$\frac{1}{d} = \E\,|\inner{e_1}{\phi}|^2 = \E\,|\inner{e_1}{\phi}|^2 \|\phi\|^2 =
\E\,|\inner{e_1}{\phi}|^2\sum_{j=1}^d|\inner{e_j}{\phi}|^2 = \alpha + (d-1)\beta.$$

We also have 
$$ 2\alpha = \E\, \left| \inner{\frac{e_1 + e_2}{\sqrt{2}}}{\phi} \right|^4 + 
 \E\, \left| \inner{\frac{e_1 - e_2}{\sqrt{2}}}{\phi} \right|^4
 = \frac{1}{2}\left(\E\, |\inner{e_1}{\phi}|^4+\E\, |\inner{e_2}{\phi}|^4 + 6 \E\, |\inner{e_1}{\phi}|^2|\inner{e_2}{\phi}|^2 \right)  = \alpha + 3\beta,$$
so that $\alpha = 3 \beta$. Solving these two equations we get $\alpha = \frac{3}{d(d+2)}$, and therefore
$\E\, |\inner{z}{\phi}|^4 = \frac{3\|z\|^4}{d(d+2)}$ for general $z \in \R^d$.

\paragraph{Acknowledgement}
The authors gratefully acknowledge Y.~Shuo Tan and R.~Vershynin for pointing out a mistake in the first version of this paper. H.J.~thanks Yuri Bakhtin and Afonso Bandeira for various valuable discussions.

\bibliographystyle{plain}
\bibliography{Nonlinear_Kaczmarz}

\end{document}